\documentclass[a4paper,12pt]{article}
\usepackage{amsfonts}
\usepackage{amssymb}
\usepackage{latexsym}
\usepackage{amsmath}
\usepackage{amsthm}
\usepackage{graphicx}
\usepackage{color}

\newtheorem{theorem}{Theorem}

\newtheorem{definition}{Definition}
\newtheorem{remark}{Remark}

\newcommand{\NN}{\mathbb{N}}

\newcommand{\bsgamma}{\boldsymbol{\gamma}}
\newcommand{\bsalpha}{\boldsymbol{\alpha}}
\newcommand{\bsx}{\boldsymbol{x}}
\newcommand{\bsone}{\boldsymbol{1}}
\newcommand{\bszero}{\boldsymbol{0}}
\newcommand{\uu}{\mathfrak{u}}
\newcommand{\cP}{\mathcal{P}}
\newcommand{\cS}{\mathcal{S}}

\newcommand{\cF}{\mathcal{F}}
\newcommand{\ee}{\mathrm{e}}
\newcommand{\rd}{\,\mathrm{d}}
\newcommand{\disc}{\mathrm{disc}}

\title{Tractability properties of the weighted star discrepancy of the Halton sequence}
\author{Aicke Hinrichs\thanks{A. Hinrichs is supported by the Austrian Science Fund (FWF) Project F5513-N26, which is a part of the Special Research Program ``Quasi-Monte Carlo Methods: Theory and Applications''.}, Friedrich Pillichshammer\thanks{F. Pillichshammer is supported by the Austrian Science Fund (FWF) Project F5509-N26, which is a part of the Special Research Program ``Quasi-Monte Carlo Methods: Theory and Applications''.}, and Shu Tezuka}

\date{}

\begin{document}

\maketitle

\begin{abstract}
We study the weighted star discrepancy of the Halton sequence. In particular, we show that the Halton sequence achieves strong polynomial tractability for the weighted star discrepancy for product weights $(\gamma_j)_{j \ge 1}$ under the mildest condition on the weight sequence known so far for explicitly constructive sequences. The condition requires $\sup_{d \ge 1} \max_{\emptyset \not= \uu \subseteq [d]} \prod_{j \in \uu} (j \gamma_j) < \infty$. The same result holds for Niederreiter sequences and for other types of digital sequences. Our results are true also for the weighted unanchored discrepancy.
\end{abstract}

\centerline{\begin{minipage}[hc]{130mm}{
{\em Keywords:} weighted star discrepancy, tractability, Halton sequence, digital sequence, quasi-Monte Carlo\\
{\em MSC 2010:} 11K38, 11K45, 65C05}
\end{minipage}}

\section{Weighted star discrepancy and tractability}

The {\it local discrepancy} of an $N$-point set ${\cal P}_d$ in $[0,1)^d$ 
is defined as $$\Delta_{\cP_d}(\bsalpha):=\frac{1}{N}\sum_{\bsx \in \cP_d} \bsone_{[\bszero,\bsalpha)}(\bsx)-{\rm Volume}([\bszero,\bsalpha))$$
for all $\bsalpha=(\alpha_1,\ldots,\alpha_d) \in [0,1]^d$, where $[\bszero,\bsalpha)=[0,\alpha_1)\times [0,\alpha_2)\times \ldots \times [0,\alpha_d)$ and $\bsone_{[\bszero,\bsalpha)}$ is the characteristic function of this interval.

Let $[d]=\{1,2,\ldots,d\}$ and let $$\bsgamma=\{\gamma_{\uu}\ : \ \emptyset \not=\uu \subseteq [d]\} \subseteq [0,1]$$ be a given set of positive weights.

\begin{definition}[Weighted star discrepancy]\rm
For an $N$-point set $\cP_d$ in $[0,1)^d$ the {\em $\bsgamma$-weighted star discrepancy} is defined as
$$D_{N,{\bsgamma}}^*({\cal P}_d):=\sup_{\bsalpha\in [0,1]^d}
\max_{\emptyset\ne {\mathfrak u}\subseteq [d]} \gamma_{\mathfrak
  u}|\Delta_{\cP_d}((\bsalpha_{\mathfrak u},\bsone))|,$$
where for $\bsalpha=(\alpha_1,\ldots,\alpha_d) \in [0,1]^d$ and for
$\uu \subseteq [d]$ we put $(\bsalpha_{\uu},\bsone)=(y_1,\ldots,y_d)$
with $y_j=\alpha_j$ if $j \in \uu$ and $y_j=1$ if $j \not\in \uu$. 
\end{definition}

\begin{remark}\rm
For $\emptyset \not= \uu \subseteq [d]$ let $\cP_d(\uu)$ be the $|\uu|$-dimensional point set consisting of the projection of the points in $\cP_d$ to the components given in $\uu$. Then we have (see \cite[Lemma~1]{P17})
\begin{equation}\label{fo:wdisc1}
D_{N,{\bsgamma}}^*({\cal P}_d) = \max_{\emptyset\ne {\mathfrak u}\subseteq [d]} \gamma_{\uu} D_N^{\ast}(\cP_d(\uu)).
\end{equation}
In some papers \eqref{fo:wdisc1} is used as definition of the weighted star discrepancy, e.g., in \cite{HSW}. 
\end{remark}

If $\gamma_{[d]}=1$ and $\gamma_{\uu}=0$ for all $\uu \varsubsetneqq [d]$, or likewise, if $\gamma_{\uu}=1$ for all $\uu \subseteq [d]$, then we obtain the classical, i.e., unweighted star discrepancy $D_N^*(\cP_d)$ which we simply call star discrepancy. A popular choice for the weights are {\it product weights} given by a non-increasing sequence $(\gamma_j)_{j \ge 1}$ of positive reals, i.e., $\gamma_1\ge \gamma_2 \ge \gamma_3\ge \ldots$. Then for $\emptyset \not=\uu \subseteq [d]$ one defines 
\begin{equation}\label{def:prodweight}
\gamma_{\uu}=\prod_{j\in \uu} \gamma_j.
\end{equation}

The $\bsgamma$-weighted star discrepancy of an $N$-point set $\cP_d$ in $[0,1)^d$, introduced by Sloan and Wo\'{z}niakowski \cite{SW98}, is intimately linked to the worst-case integration error of quasi-Monte Carlo (QMC) rules of the form $$\frac{1}{N} \sum_{\bsx \in \cP_d} f(\bsx)$$ for functions $f$ from the weighted function class $\mathcal{F}_{d,1,\bsgamma}$, which is given as follows: Let $\mathcal{W}_1^{(1,1,\ldots,1)}([0,1]^d)$ be the Sobolev space of functions defined on $[0,1]^d$ that are once differentiable in each variable, and whose derivatives have finite $L_1$ norm. Then $$\cF_{d,1,\bsgamma}=\{f \in \mathcal{W}_1^{(1,1,\ldots,1)}([0,1]^d) \ : \ \|f\|_{d,1,\bsgamma}< \infty\},$$ where $$\|f\|_{d,1,\bsgamma} =|f(\bsone)| + \sum_{\emptyset \not=\uu \subseteq [d]} \frac{1}{\gamma_\uu}\left\|\frac{\partial^{|\uu|}}{\partial \bsx_{\uu}}f(\bsx_{\uu},\bsone)\right\|_{L_1}.$$ The fundamental error estimate is a weighted version of the Koksma-Hlawka inequality, see \cite[p.~65]{NW10}. In fact, the worst-case error of a QMC rule in $\cF_{d,1,\bsgamma}$ is exactly the $\bsgamma$-weighted star discrepancy of the point set used in the QMC rule:

$$\sup_{\|f\|_{d,1,\bsgamma} \le 1} \left|\int_{[0,1]^d} f(\bsx) \rd \bsx - \frac{1}{N} \sum_{\bsx \in \cP_d} f(\bsx)\right|=D_{N,{\bsgamma}}^*({\cal P}_d).$$

In the classical theory one studies the dependence of the star discrepancy on the number $N$ of elements of a point set in a fixed dimension $d$, see, e.g., the books \cite{BC,kuinie,LP14,Mat99,niesiam}. The dependence of the star discrepancy on the dimension $d$ is the subject of tractability studies. We now introduce the necessary background. For $d,N \in \NN$ the $N^{{\rm th}}$ minimal weighted star discrepancy is defined as 
$$\disc_{\bsgamma}(N,d) = \inf_{\cP \subseteq [0,1)^d \atop |\cP|=N} D_{N,\bsgamma}^{\ast}(\cP).$$
We would like to have a point set in the $d$-dimensional unit cube with weighted star discrepancy at most $\varepsilon \in (0, 1)$ and we are looking for the smallest cardinality
$N$ of a point set such that this can be achieved. For $\varepsilon \in (0, 1)$ and $d \in \NN$ we define the {\it inverse of the weighted star discrepancy} (or, in a wider context, the {\it information complexity}) as
$$N_{\min}(\varepsilon,d) := \min\{N \in \NN \ : \ \disc_{\bsgamma}(N, d) \le \varepsilon\}.$$

One is now interested in the behavior of the inverse of the weighted star discrepancy $N_{\bsgamma}(\varepsilon,d)$ for $\varepsilon \rightarrow 0$ and $d \rightarrow \infty$. This is the subject of tractability. An overview on the current state of the art of tractability theory can be found in the three volumes \cite{NW08,NW10,NW12}. Here we study the concept of polynomial tractability. Informally, polynomial tractability means that there exists an $N$-point set with $N$ depending polynomially on $d$ and $\varepsilon^{-1}$ such that the weighted star discrepancy of this point set is bounded by $\varepsilon$. 

\begin{definition}[Tractability]\rm
The weighted star discrepancy is said to be
\begin{enumerate}
\item {\it polynomially tractable}, if there exist non-negative real numbers $C, \alpha$ and $\beta$ such that
$$N_{\min}(\varepsilon,d) \le C d^{\alpha}\varepsilon^{-\beta}\ \ \ \mbox{ for all $d\in \NN$ and for all $\varepsilon \in (0,1)$.}$$ 
\item {\it strongly polynomially tractable}, if there exist non-negative real numbers $C$ and $\beta$ such that
\begin{equation}\label{defspt}
N_{\min}(\varepsilon,d) \le C \varepsilon^{-\beta}\ \ \ \mbox{ for all $d\in \NN$ and for all $\varepsilon \in (0,1)$.}
\end{equation}
The infimum over all $\beta > 0$ such that \eqref{defspt} holds is called the $\varepsilon$-exponent of strong polynomial tractability.
\end{enumerate}
\end{definition}

In the following we give a brief survey about known results on tractability of the weighted star discrepancy, where we will distinguish between ``existence results'' and ``constructive results''. 

Before we do so, we shall specify more exactly what we understand by the intuitive notion of   ``constructive result''. We are aware that most of the following existence results can be made ``constructive'' in the sense that one can compute suitable point sets with finitely many arithmetic operations. This can be achieved by means of diverse derandomization methods. However, in most of these cases the computational effort to find such point sets explicitly grows exponential in $d$. In this sense we understand by a ``constructive result'' that the corresponding point set can be found or constructed by a polynomial-time algorithm in $d$ and in $\varepsilon^{-1}$. We refer also to \cite[Section~4.3]{DGS2005} for a discussion in this direction.

\paragraph{Existence results.} From Heinrich, Novak, Wasilkowski, and Wo\'{z}niakowski~\cite{hnww} it is known that there exists an absolute constant $C>0$ such that  
\begin{equation}\label{HNWW}
\disc_{\bsone}(N,d) \le C \sqrt{\frac{d}{N}} \ \ \ \mbox{for all $d,N \in \NN$,}
\end{equation}
where $\bsone$ is the constant sequence $(1)_{j \ge 1}$ (Aistleitner~\cite{Aist} showed that one can choose $C=10$). Hence the classical star discrepancy is polynomially tractable with $\varepsilon$-exponent at most 2. From Hinrichs \cite{hin} we know that the inverse of the classical star discrepancy is at least $c d \varepsilon^{-1}$ with an absolute constant $c>0$ for all $\varepsilon \in (0,\varepsilon_0]$ and $d \in \NN$. From these results it follows that the classical star discrepancy cannot be strongly polynomially tractable. For a derandomization of the result \eqref{HNWW} see \cite{DGS2005}.

For the weighted star discrepancy we know from \cite{HPS2008} that there exists an absolute constant $C>0$ such that 
\begin{equation}\label{bdx1}
\disc_{\bsgamma}(N,d) \le C\,\frac{1+\sqrt{\log d}}{\sqrt{N}}  \max_{\emptyset \not=\uu \subseteq [d]} \gamma_{\uu} \sqrt{|\uu|}  \ \ \ \mbox{for all $d,N \in \NN$.}
\end{equation}
The proof of this result is based on \eqref{HNWW}. If $\sup_{d \ge 1}\max_{\emptyset \not=\uu \subseteq [d]} \gamma_{\uu} \sqrt{|\uu|} < \infty$, then \eqref{bdx1} implies polynomial tractability with $\varepsilon$-exponent 2, see \cite{HPS2008} for details. A slightly improved and numerically explicit version of \eqref{bdx1} can be found in \cite[Theorem~1]{Aist2}.

Hickernell, Sloan, and Wasilkowski  \cite{HSW} considered product weights $(\gamma_j)_{j=1}^{\infty}$ and showed that if there exists some $a>0$ such that
\begin{equation}\label{cond8}
\sum_{j=1}^{\infty} \gamma_j^a < \infty, 
\end{equation}
then for every $\delta>0$ there exists some $C(\delta)>0$ such that 
$$\disc_{\bsgamma}(N,d) \le \frac{C(\delta)}{N^{1/2-\delta}} \ \ \ \mbox{for all $d,N \in \NN$.}$$
A typical instance for weights satisfying condition \eqref{cond8} is $\gamma_k=O(k^{-\tau})$ for an arbitrary small number $\tau>0$. 

This result was further improved by Aistleitner \cite{Aist2} who showed that for product weights satisfying the condition 
\begin{equation}\label{condA}
\sum_{j=1}^\infty {\rm e}^{-c \gamma_j^{-2}} < \infty
\end{equation}
for some $c > 0$ there is a constant $C_{\bsgamma}>0$ such that 
$$\disc_{\bsgamma}(N,d) \le \frac{C_{\bsgamma}}{\sqrt{N}} \ \ \ \mbox{for all $d,N \in \NN$.}$$ 
Consequently, the weighted star discrepancy for such weights is strongly polynomially tractable, with $\varepsilon$-exponent at most 2. A typical sequence $(\gamma_j)_{j \ge 1}$ satisfying condition \eqref{condA} is $\gamma_j=\widehat{c}/\sqrt{\log j}$ for some $\widehat{c}>0$. Currently, this is the mildest condition on the weights in order to achieve strong polynomial tractability for the weighted star discrepancy.

\paragraph{Constructive results.} In \cite{DLP2006} the authors considered digital nets and showed that for every prime number $p$, every $m \in \NN$ and for given product weights $(\gamma_j)_{j \ge 1}$ with 
\begin{equation}\label{sumwei}
\sum_{j = 1}^{\infty} \gamma_j < \infty
\end{equation}
one can construct component-by-component an $N$-point set $\cP$ with $N=p^m$ in $[0,1)^d$ such that for every $\delta >0$ there exists some $C_{\bsgamma,\delta}>0$ with the property $$D_{N,\bsgamma}^{\ast}(\cP)\le \frac{C_{\bsgamma,\delta}}{p^{m(1-\delta)}}.$$ This result implies that the weighted star discrepancy of the CBC-constructed point sets achieves strong polynomial tractability with $\varepsilon$-exponent equal to 1, as long as the weights $\gamma_j$ are summable. The so-called fast CBC algorithm to construct a suitable generating vector of the digital point set $\cP$ requires $O(d N \log N)$ arithmetic operations. This implies that $\cP$ can be found by a polynomial-time algorithm in $d$ and in $\varepsilon^{-1}$. Because of the CBC-construction the point set $\cP$ depends on the weights $\bsgamma$. See also \cite{DLP2006,dnp06,DP10,HPS2008} for more details.

The summability condition \eqref{sumwei} on the weights appears also in \cite{DP14} where so-called Korobov $p$-sets are studied. If \eqref{sumwei} holds, the weighted star discrepancy of these $p$-sets also achieves strong polynomial tractability but with the weaker $\varepsilon$-exponent 2. The advantage here is that the point sets are really explicit and do not need to be constructed componentwise. This also means that these point sets are universal in the sense that they are independent of the weights $\bsgamma$.

Beside the results for $p$-sets the following constructive and universal results are known: Wang \cite[Lemma~1]{wang2} proved that for the initial $N$ elements of a Niederreiter sequence $\cS_{d}$ in prime-power base $q$ (see \cite{DP10,niesiam} for a definition) for every $\uu \subseteq [d]$ it holds that $$D_N^{\ast}(\cS_{d}(\uu)) \le \frac{1}{N} \prod_{j \in \uu}(C j \log(j+q) \log(qN)),$$ where $C>0$ is an absolute constant which is independent of $\uu$ and $d$. Similar results can be shown for Sobol' sequences and for the Halton sequence (see \cite{wang1,wang2}). From this result one obtains
$$D_{N,\bsgamma}^{\ast}(\cS_{d}) \le \frac{1}{N} \max_{\emptyset \not= \uu \subseteq [d]} \gamma_{\uu} \prod_{j \in \uu}  (C j \log(j+q) \log(qN)).$$ In the case of product weights this implies that the weighted star discrepancy of the Niederreiter sequence achieves strong polynomial tractability with $\varepsilon$-exponent 1 whenever the weights satisfy $\sum_{j \ge 1} (j \log j) \gamma_j  < \infty$. The same result can be shown for Sobol' sequences and for the Halton sequence (see Section~\ref{secHalt}). We will improve these results.

In this paper we study the Halton sequence in more detail and present the currently mildest condition on product weights under which a constructive and universal result for strong polynomial tractability for the weighted star discrepancy holds. Similar results hold for special kinds of digital sequences, e.g., Niederreiter sequences, Xing-Niederreiter sequences, Hofer-Niederreiter sequences and Sobol' sequences. The results are presented and proved in the following section. A brief discussion of the obtained discrepancy bounds is given in Section~\ref{discussion}.

\section{The Halton sequence}\label{secHalt}

Let $\mathcal{H}_{b_1,\ldots,b_d}$ be the $d$-dimensional Halton sequence in pairwise coprime bases $b_1,\ldots,b_d$. Throughout this paper we assume that $b_1,b_2,b_3,\ldots$ are the prime numbers in increasing order. The $n^{{\rm th}}$ element of the Halton sequence is given by $$\bsx_n=(\varphi_{b_1}(n),\ldots,\varphi_{b_d}(n))$$ where, for some integer $b>1$ and $n$ with $b$-adic expansion $n=n_0+n_1 b+ n_2 b^2+\cdots$ we define $$\varphi_b(n)=\frac{n_0}{b}+\frac{n_1}{b^2}+\frac{n_2}{b^3}+\cdots.$$

\begin{remark}\rm
Note that the first $N$ elements of the $d$-dimensional Halton sequence can be computed in $O(dN)$ arithmetic operations. This follows from the observation that, given the first $b^k$ elements $x_0,x_1,\ldots,x_{b^k-1}$ of the $b$-adic van der Corput sequence $\mathcal{H}_b$ (i.e. the 1-dimensional Halton sequence in base $b$), one obtains the initial $b^{k+1}$ elements of $\mathcal{H}_b$ by computing  
$$\begin{array}{llll}
x_0, & x_1, & \ldots & x_{b^k-1}, \\
x_0+\frac{1}{b^{k+1}}, & x_1+\frac{1}{b^{k+1}}, & \ldots & x_{b^k-1}+\frac{1}{b^{k+1}},\\
x_0+\frac{2}{b^{k+1}}, & x_1+\frac{2}{b^{k+1}}, & \ldots & x_{b^k-1}+\frac{2}{b^{k+1}},\\
\multicolumn{4}{l}\dotfill\\
x_0+\frac{b-1}{b^{k+1}}, & x_1+\frac{b-1}{b^{k+1}}, & \ldots & x_{b^k-1}+\frac{b-1}{b^{k+1}}.
\end{array}$$
\end{remark}

The weighted star discrepancy of the Halton sequence for finite order weights has been studied in \cite{T16}. Here we consider product weights.

It follows from \cite[Theorem~3.6]{niesiam}, see also \cite[Eq. (20)]{wang1} that for $\emptyset \not=\uu \subseteq [d]$ we have $$D_N^{\ast}(\mathcal{H}_{b_1,\ldots,b_d}(\uu)) \le \frac{1}{N} \prod_{j \in \uu} \left(\frac{b_j -1}{2 \log b_j} \log N + \frac{b_j+3}{2}\right).$$ From this result it follows that the weighted star discrepancy of the Halton sequence achieves strong polynomial tractability if $$\sum_{j \ge 1}  (j \log j) \gamma_j < \infty.$$ 
 
It is the aim of this paper to improve this result. Thereby we give the currently best ``constructive'' proof for strong polynomial tractability of the weighted star discrepancy. This means that we give the mildest condition on product weights such that the weighted star discrepancy of an explicit point set achieves strong polynomial tractability.

\begin{theorem}\label{thm1}
Let $b_1,b_2,b_3,\ldots$ be the prime numbers in increasing order. Then we have:
\begin{itemize}
\item If $\sum_{j \ge 1} j \gamma_j < \infty$, then the weighted star discrepancy of the Halton sequence $\mathcal{H}_{b_1,\ldots,b_d}$ achieves strong polynomial tractability with $\varepsilon$-exponent 1, which is optimal. 
\item If $\sup_{d \ge 1} \max_{\emptyset \not= \uu \subseteq [d]} \prod_{j \in \uu} (j \gamma_j) < \infty$, then the weighted star discrepancy of the Halton sequence $\mathcal{H}_{b_1,\ldots,b_d}$ achieves strong polynomial tractability with $\varepsilon$-exponent at most 2. 
\end{itemize}
\end{theorem}

\begin{remark}\rm
Note that the second item of Theorem~\ref{thm1} tells us that weights $\gamma_j=\frac{1}{j}$ already guarantee strong polynomial tractability. Although this is still much more demanding than Aistleitner's condition \eqref{condA}, this result is the currently mildest weight condition for a constructive proof of strong polynomial tractability of the weighted star discrepancy. Furthermore, it is the first ``constructive'' result which does not require that the weights are summable in order to achieve strong polynomial tractability. 
\end{remark}

\begin{proof}[Proof of Theorem~\ref{thm1}]
For the proof we use Halton's bound on the star discrepancy of $\mathcal{H}_{b_1,\ldots,b_d}$ from \cite{H60} which implies that for $\emptyset \not=\uu \subseteq [d]$ we have $$D_N^{\ast}(\mathcal{H}_{b_1,\ldots,b_d}(\uu)) \le \frac{(\log N)^{|\uu|}}{N} \prod_{j \in \uu} \frac{3 b_j -2}{\log b_j}.$$ Since $b_j$ is the $j^{{\rm th}}$ largest prime number we obtain (see \cite{T16}) $$\frac{3 b_j-2}{\log b_j} \le 6 j,$$ and hence 
\begin{equation}\label{discPu}
D_N^{\ast}(\mathcal{H}_{b_1,\ldots,b_d}(\uu)) \le \frac{(\log N)^{|\uu|}}{N} \prod_{j \in \uu} 6 j.
\end{equation}
This implies that
\begin{eqnarray}\label{bd_disc_direct}
D_{N,\bsgamma}^{\ast}(\mathcal{H}_{b_1,\ldots,b_d}) & \le & \frac{1}{N} \max_{\emptyset \not= \uu \subseteq [d]} \prod_{j \in \uu} (6 j \gamma_j \log N)\\
& \le &   \frac{1}{N} \sum_{\emptyset \not= \uu \subseteq [d]} \prod_{j \in \uu} (6 j \gamma_j \log N) \label{maxsumchange}\\
& = & \frac{1}{N} \left(-1+ \prod_{j=1}^d (1+6 j \gamma_j \log N)\right).\nonumber
\end{eqnarray}
Assume that $$\sum_{j \ge 1} j \gamma_j < \infty.$$ Then, using an argumentation presented in \cite[Lemma~3]{HN} (see also \cite[p. 222]{DP10}), it follows that for every $\delta>0$ there exists a $c_{\delta}>0$ such that $$\prod_{j=1}^d (1+6 j \gamma_j \log N) < c_{\delta} N^{\delta}.$$ This implies 
\begin{equation*}
D_{N,\bsgamma}^{\ast}(\mathcal{H}_{b_1,\ldots,b_d}) \le \frac{c_{\delta}}{N^{1-\delta}}.
\end{equation*}
Now, if $N \ge \lceil (c_{\delta} \varepsilon^{-1})^{1/(1-\delta)} \rceil$ we obtain $D_{N,\bsgamma}^{\ast}(\mathcal{H}_{b_1,\ldots,b_d}) \le \varepsilon$ and this implies that the weighted star discrepancy of the Halton sequence achieves strong polynomial tractability. Since $\delta>0$ can be chosen arbitrary closely to zero we find that the $\varepsilon$-exponent equals 1. Thus, the first item is shown. 

Now we prove the second item: we use the trivial fact that the star discrepancy is always bounded by 1 which leads to an improved version of the estimate \eqref{bd_disc_direct}. We have
\begin{eqnarray*}
D_{N,\bsgamma}^{\ast}(\mathcal{H}_{b_1,\ldots,b_d}) & \le & \max_{\emptyset \not= \uu \subseteq [d]} \prod_{j \in \uu} \gamma_j \min\left\{1, \frac{(\log N)^{|\uu|}}{N} \prod_{j \in \uu}(6 j)\right\}\\
& = & \max_{\emptyset \not= \uu \subseteq [d]} \prod_{j \in \uu} (j \gamma_j) \min\left\{\prod_{j \in \uu} \frac{1}{j}, \frac{(6 \log N)^{|\uu|}}{N}\right\}.
\end{eqnarray*}
For $\emptyset \not=\uu \subseteq [d]$ with $|\uu|=\ell$ we have $$\min\left\{\prod_{j \in \uu} \frac{1}{j}, \frac{(6 \log N)^{|\uu|}}{N}\right\} \le \min\left\{\frac{1}{\ell!}, \frac{(6 \log N)^{\ell}}{N}\right\}\le \min\left\{\left(\frac{\ee}{\ell}\right)^{\ell}, \frac{(6 \log N)^{\ell}}{N}\right\}$$ where we used Stirling's formula for the last inequality. 

We consider $\ell$ as a real variable and determine $\ell$ for which 
\begin{equation}\label{eq1}
\left(\frac{\ee}{\ell}\right)^{\ell} = \frac{(6 \log N)^{\ell}}{N}.
\end{equation}
To this end we use Lambert's $W$-function $W(z)$ which is the inverse function of the function $x \mapsto x \ee^x$ and which satisfies $W(z) \ee^{W(z)}=z$. With the help of $W$ we can solve the equation $(ax)^x =z$ for given $a,z$. The solution is given by $x=\frac{\log z}{W(\log z^a)}$, since for this choice and by using the equation $$\ee^{W(\log z^a)}=\frac{\log z^a}{W(\log z^a)}=a x$$ we obtain $$(a x)^x = \left(\ee^{W(\log z^a)}\right)^{\frac{\log z}{W(\log z^a)}}=\ee^{\log z}=z.$$ 

For our specific equation \eqref{eq1} we therefore obtain the solution $$\ell^*=\frac{\log N}{W\left(\frac{6}{\ee} (\log N)^2\right)}.$$

Hence, 
\begin{eqnarray}
\min\left\{\left(\frac{\ee}{\ell}\right)^{\ell}, \frac{(6 \log N)^{\ell}}{N}\right\} & \le & \frac{1}{N} (6 \log N)^{\frac{\log N}{W\left(\frac{6}{\ee} (\log N)^2\right)}}\nonumber\\ 
&= &\frac{1}{N} \exp\left(\frac{\log N \log \log N^6}{W\left(\frac{6}{\ee} (\log N)^2\right)}\right)\nonumber\\
& = & \frac{1}{N} N^{\frac{\log \log N^6}{W\left(\frac{6}{\ee} (\log N)^2\right)}}\label{expo1}
\end{eqnarray}
We have $$W(x) \approx \log x - \log \log x + \frac{\log \log x}{\log x} \ \ \ \mbox{ and }\ \ \ W(x)\ge \log x - \log \log x \ \ \mbox{for $x \ge {\rm e}$}.$$ This yields  $$W\left(\frac{6}{\ee} (\log N)^2\right) \ge \log\left(\frac{6}{\ee} (\log N)^2\right)-\log\log \left(\frac{6}{\ee} (\log N)^2\right).$$
Hence, for the exponent in \eqref{expo1} we have 
\begin{eqnarray*}
\frac{\log \log N^6}{W\left(\frac{6}{\ee} (\log N)^2\right)} \le \frac{\log \log N +\log 6}{2 \log\log N+\log 6 -1-\log(2 \log\log N+\log 6 -1)}=:\delta^*(N).
\end{eqnarray*}
This implies that $$D_{N,\bsgamma}^{\ast}(\mathcal{H}_{b_1,\ldots,b_d}) \le  \frac{1}{N^{1-\delta^*(N)}} \max_{\emptyset \not= \uu \subseteq [d]} \prod_{j \in \uu} (j \gamma_j).$$
Note that $\lim_{N \rightarrow \infty}\delta^*(N)=\frac{1}{2}$ and hence, for every $\delta>0$ there exists some $C_{\delta}>0$ such that 
\begin{equation}\label{bd_di_ha2}
D_{N,\bsgamma}^{\ast}(\mathcal{H}_{b_1,\ldots,b_d}) \le \frac{C_{\delta}}{N^{1/2-\delta}} \max_{\emptyset \not= \uu \subseteq [d]} \prod_{j \in \uu} (j \gamma_j).
\end{equation}
If $$\sup_{d \ge 1} \max_{\emptyset \not= \uu \subseteq [d]} \prod_{j \in \uu} (j \gamma_j)< \infty$$ we obtain as above that the weighted star discrepancy of the Halton sequence achieves strong polynomial tractability. Here we can only guarantee a $\varepsilon$-exponent of at most 2.
\end{proof}

\begin{remark}\rm
The results from Theorem~\ref{thm1} hold true for all sequences that satisfy a star discrepancy bound of the form \eqref{discPu} for all projections of the sequence. Examples are the Niederreiter sequences $\cS_b$ in base $b$ which satisfy $$D_N^{\ast}(\cS_b(\uu)) \le \frac{(\log N)^{|\uu|}}{N} \left( \frac{4 b^2}{\log b}\right)^{|\uu|} \prod_{j \in \uu} j,$$ see \cite[Corollary~2]{T16}, or Xing-Niederreiter sequences and Hofer-Niederreiter sequences 
in base $b$ and genus $g$ which satisfy $$D_N^{\ast}(\cS_b(\uu)) \le b^g \frac{(\log N)^{|\uu|}}{N} C^{|\uu|} \prod_{j \in \uu} j$$ for some $C>1$, see \cite[Corollary~5]{T16}.

For the Sobol' sequence $\cS^{{\rm Sob}}$ in base 2 we have $$D_N^{\ast}(\cS^{{\rm Sob}}(\uu)) \le  \frac{(\log N)^{|\uu|}}{N} C^{|\uu|} \prod_{j \in \uu} (j \log_2 \log_2(j+3)),$$ with some $C>1$, see \cite[Corollary~3]{T16}. With the same methods as above one can show that the weighted star discrepancy of the Sobol' sequence achieves strong polynomial tractability if $$\sum_{j\ge 1} (j \log_2 \log_2(j+3) \gamma_j)< \infty.$$ The $\varepsilon$-exponent is again 1 in this case. Furthermore, if $$\sup_{d \ge 1} \max_{\emptyset \not= \uu \subseteq [d]} \prod_{j \in \uu} (j \log_2\log_2(j+3) \gamma_j) < \infty,$$ then the weighted star discrepancy of the Sobol' sequence achieves strong polynomial tractability with $\varepsilon$-exponent at most 2. 
\end{remark}

\begin{remark}\rm
Note that Theorem~\ref{thm1} can be generalized to the weighted
unanchored discrepancy discussed in \cite{HSW}, which is defined as
$$
D_{N,{\bsgamma}}({\cal P}_d):=\max_{\emptyset\ne {\mathfrak u}\subseteq
[d]} \gamma_{\mathfrak  u} D_{N}({\cal P}_d(\uu)),
$$
where
$D_{N}({\cal P}_d(\uu))$ is the unanchored discrepancy
of the $|\uu|$-dimensional  point set ${\cal P}_d(\uu)$.
Since  
$$
D_{N}({\cal P}_d(\uu)) \le 2^{|\uu|}D^*_{N}({\cal P}_d(\uu)),
$$
we have for the Halton sequence,
$$
D_N(\mathcal{H}_{b_1,\ldots,b_d}(\uu)) \le \frac{(\log N)^{|\uu|}}{N} \prod_{j \in \uu} 12 j.
$$
Therefore, with all parts of the proof remaining unchanged except for
the constant 6 replaced by 12, we see that Theorem~\ref{thm1} still
holds for the unanchored case.
\end{remark}

\section{Discussion of the results}\label{discussion}

We have shown that, formally, the Halton sequence achieves strong polynomial tractability for the weighted star discrepancy for sufficiently fast decaying weights. If $\sum_{j \ge 1} j \gamma_j < \infty$ we even obtain the optimal $\varepsilon$-exponent 1, if $\sup_{d \ge 1} \max_{\emptyset \not= \uu \subseteq [d]} \prod_{j \in \uu} (j \gamma_j) < \infty$ we still obtain an $\varepsilon$-exponent of at most $2$. This seems to be excellent. However, the problem with our seemingly excellent bounds is that in some cases the involved constants are astronomically large, especially when the weights do not decrease very fast.

Consider, for example, weights of the form $\gamma_j=1/j^{1+\alpha}$ with $\alpha>1$ which guarantees that $\sum_{j \ge 1} j \gamma_j < \infty$. Then we showed $$D_{N,\bsgamma}^{\ast}
(\mathcal{H}_{b_1,\ldots,b_d}) \le \frac{c_{\delta}}{N^{1-\delta}}.$$ 
If we follow the proof of \cite[Lemma~3]{HN} it turns out, that $$c_{\delta}=\left(1+\frac{1}{\sigma_w}\right)^w \ \ \ \mbox{ where }\ \ \ \sigma_w=6 \sum\limits_{j=w+1}^{\infty} j \gamma_j = 6 \sum\limits_{j=w+1}^{\infty} \frac{1}{j^{\alpha}},$$ and where $w$ is as large such that 
\begin{equation}\label{cond1}
\sigma_w \le \frac{\delta}{1+\sigma_0}.
\end{equation}
By elementary estimates of $\sum_{j \ge 1} j^{-\alpha}$ one obtains that \eqref{cond1} implies $$w \ge -1+\left(\frac{6}{(\alpha-1) \delta} \left(1+\frac{6}{\alpha-1}\right)\right)^{\frac{1}{\alpha-1}}\ \ \mbox{ and  }\ \ 
c_{\delta} \ge \left(1+\frac{(\alpha-1) w^{\alpha-1}}{6}\right)^w.$$ For example, if $\delta=0.1$ and $\alpha=1.1$, then $w \ge -1+ (600 (1+60))^{10} =4.31331 \times 10^{45}$ and it is even impossible to compute the lower bound on $c_{\delta}$. In the following table we collect some values for $c_{\delta}$:

\[
\begin{array}{c||c|c|c||l}
\delta &  0.9  &  0.5  & 0.1  & \\
\hline\hline
c_{\delta} & 4 \times 10^{35714}     & 10^{139333}       &  10^{5152589}     & \alpha=1.5\\
\hline
c_{\delta} & 5 \times 10^{42}      &  1,6 \times 10^{97}        &  1,7 \times 10^{775}      & \alpha=2\\
\hline
c_{\delta} & 24.5      &  1129.5       &  1.7 \times 10^{15}      & \alpha=3\\
\hline
c_{\delta} & 1.29      &  2.5       &  1922      & \alpha=4\\
\end{array}
\]

One might think that the estimate of the maximum by a sum in \eqref{maxsumchange} may be the reason for these poor values for $c_{\delta}$. This is not the case.  We show that for weights $\gamma_{j}=1/j^{1+\alpha}$ with $\alpha>1$ we can avoid the replacement of the maximum in \eqref{bd_disc_direct} by the sum over all $\emptyset\not=\uu \subseteq [d]$ in \eqref{maxsumchange}. To this end we use \eqref{bd_disc_direct} and the shorthand notation $x:=(6 \log N)^{1/\alpha}$ to obtain
\begin{eqnarray}\label{bd_comp}
D_{N,\bsgamma}^{\ast}(\mathcal{H}_{b_1,\ldots,b_d}) \le \frac{1}{N} \max_{\emptyset \not= \uu \subseteq [d]}  \prod_{j \in \uu} \frac{6 \log N}{j^{\alpha}}=\frac{1}{N} \max_{k=1,\ldots,d} \frac{(6 \log N)^k}{(k!)^{\alpha}} =\frac{1}{N} \left(\max_{k=1,\ldots,d} \frac{x^k}{k!} \right)^{\alpha}.
\end{eqnarray}

We have $$\frac{\frac{x^k}{k!}}{\frac{x^{k+1}}{(k+1)!}}=\frac{k+1}{x} \left\{ 
\begin{array}{ll}
>1 & \mbox{if $k>x-1$},\\
=1 & \mbox{if $k=x-1$},\\
<1 & \mbox{if $k<x-1$}.
\end{array}\right.$$

Hence, writing $y=\lceil x \rceil \in \NN$ and using Stirling's formula we obtain
\begin{eqnarray*}
\max_{k=1,\ldots,d} \frac{x^k}{k!} & \le & \max_{k\ge 1} \frac{x^k}{k!}=\frac{x^{\lceil x \rceil -1}}{(\lceil x \rceil -1)!} = \frac{\lceil x \rceil}{x} \frac{x^{\lceil x \rceil}}{\lceil x \rceil !}\le 2 \frac{y^y}{y!} \le 2 \frac{{\rm e}^y}{\sqrt{2 \pi y}}.
\end{eqnarray*}
Thus we have
\begin{eqnarray*}
D_{N,\bsgamma}^{\ast}(\mathcal{H}_{b_1,\ldots,b_d}) & \le & \frac{2^{\alpha}}{N} \left(\frac{1}{2 \pi (6 \log N)^{1/\alpha}}\right)^{\alpha/2}{\rm e}^{((6 \log N)^{1/\alpha}+1)\alpha}\\
& = & \left(\frac{2 {\rm e}^2}{\pi}\right)^{\alpha/2} \frac{1}{\sqrt{6 \log N}} {\rm e}^{\alpha(6 \log N)^{1/\alpha} - \log N}\\
& = & \left(\frac{2 {\rm e}^2}{\pi}\right)^{\alpha/2} \frac{1}{\sqrt{6 \log N}} \frac{1}{N^{1-\frac{\alpha (6 \log N)^{1/\alpha}}{\log N}}}.
\end{eqnarray*}
For $\alpha>1$ we have $$\lim_{N \rightarrow \infty}\frac{\alpha(6 \log N)^{1/\alpha}}{\log N}=0.$$ Hence, for every $\delta>0$ there exists a $c_{\delta}>0$ such that $$D_{N,\bsgamma}^{\ast}(\mathcal{H}_{b_1,\ldots,b_d}) \le \frac{c_{\delta}}{N^{1-\delta}}.$$ 

In order to study the order of magnitude of $c_\delta$ we write $N={\rm e}^x$. Then
$$D_{N,\bsgamma}^{\ast}(\mathcal{H}_{b_1,\ldots,b_d}) \le \left(\frac{2 {\rm e}^2}{\pi}\right)^{\alpha/2} \frac{1}{\sqrt{6x}} {\rm e}^{\alpha (6 x)^{1/\alpha}-\delta x} \frac{1}{N^{1-\delta}} \le \left(\frac{2 {\rm e}^2}{\pi}\right)^{\alpha/2} \frac{1}{\sqrt{6}} \max_{x \ge 1} {\rm e}^{\alpha (6 x)^{1/\alpha}-\delta x} \frac{1}{N^{1-\delta}}.$$
Now $$\max_{x \ge 1} {\rm e}^{\alpha (6 x)^{1/\alpha}-\delta x}={\rm e}^{\alpha (6 x_0)^{1/\alpha}-\delta x_0} \ \ \ \mbox{where $x_0={\rm e}^\frac{\log 6 -\alpha \log \delta}{\alpha -1}$.}$$ Hence $$c_{\delta} = \left(\frac{2 {\rm e}^2}{\pi}\right)^{\alpha/2} \frac{1}{\sqrt{6}}{\rm e}^{\alpha (6 x_0)^{1/\alpha}-\delta x_0}$$ which is again astronomically large especially when $\alpha$ is close to 1.\\

The same problem appears with the $C_{\delta}$ in \eqref{bd_di_ha2}. Note that the convergence of $\delta^*(N)$ to $1/2$ for $N \rightarrow \infty$ is very very slow. For example, in order to have $\delta^*(N) < 1$ we require $N \approx 100.000.000$. This leads to an astronomically large value for the constant $C_\delta$ in the discrepancy bound \eqref{bd_di_ha2}.

\paragraph{Open problem.} Improve the given discrepancy bounds with respect to the involved constants.

\paragraph{Acknowledgment.} We thank \'{I}sabel Pir\v{s}i\'{c} for valuable discussions concerning the computational complexity of an efficient computation of the van der Corput sequence.

\end{document}